\documentclass{amsart}
\usepackage{amsmath,amsthm,amssymb,IMjournal}

\usepackage{mathtools}
\usepackage[utf8]{inputenc}
\usepackage{enumitem}
\usepackage{graphicx}
\usepackage{xcolor}

\usepackage[hidelinks]{hyperref}
\hypersetup{pdftitle={Rigorous verification of period-doubling bifurcations in maps},
            pdfauthor={Olivier Henot, Jean-Philippe Lessard}}

\def\xadresa{{\small\it Authors' addresses:}\ \rm}
\newcommand{\R}{\mathbb{R}}
\newcommand{\C}{\mathbb{C}}
\newcommand{\Df}{D_x f}
\newcommand{\Dff}{D_x^2 f}
\newcommand{\norm}[1]{\left\|#1\right\|}
\newcommand{\Fpd}{F_{\mathrm{pd}}}
\newcommand{\nuc}{\nu_{\mathrm{c}}}
\newcommand{\dnuc}{\dot{\nu}_{\mathrm{c}}}
\newcommand{\Mat}{\mathrm{Mat}}
\newcommand{\ball}[2]{B_{#1}(#2)}

\newcommand{\bydef}{\,\stackrel{\mbox{\tiny\textnormal{\raisebox{0ex}[0ex][0ex]{def}}}}{=}\,}

\begin{document}

\title{Rigorous verification of period-doubling bifurcations in maps}

\author{\|Olivier |H\'enot|,
        \|Jean-Philippe |Lessard|}
\thanks
The research of O.~H\'enot was supported by the National Science and Technology Council (NSTC), Taiwan, under grant No.~115-2115-M-002-001-MY2. The research of J.-P.~Lessard was supported by the Natural Sciences and Engineering Research Council of Canada (NSERC).
\endthanks


\abstract
We present a computer-assisted approach for proving period-doubling bifurcations in one-parameter families of maps $x \mapsto f(x,\lambda)$ on $\R^n$.
We solve a square augmented system, called the \emph{period-doubling map}, whose unknowns are the fixed point, the parameter, the right eigenvector attached to the eigenvalue $-1$, the left eigenvector with its eigenvalue, and the vector which describes the center manifold to second order.
A contraction argument gives a solution close to a numerical approximation, at which the Jacobian is invertible.
The determinant of this Jacobian factors into a simple expression, so that it is nonzero exactly when $1$ is not an eigenvalue and $-1$ is a simple eigenvalue crossing the unit circle transversally.
Apart from the sign of the cubic normal-form coefficient, computed afterwards, it remains to check with Gershgorin disks that no other eigenvalue lies on the unit circle.
The method is applied to a $10$-mode Galerkin projection of the Kot--Schaffer growth-dispersal model, for which we prove a supercritical period-doubling bifurcation.
\endabstract



\section{Introduction}\label{sec:intro}

A period-doubling bifurcation, also called a flip bifurcation, is a codimension-one bifurcation at which a simple eigenvalue of the Jacobian matrix crosses $-1$, with no other eigenvalues on the unit circle, and a branch of period-$2$ orbits emerges from a branch of fixed points.
A cascade of such bifurcations is a well-known route to chaos, with the universal scaling discovered by Feigenbaum~\cite{Feigenbaum1978}.
In practice, period-doubling bifurcations are detected by numerical continuation.
A rigorous proof then necessitates locating the bifurcation point and verifying that the critical eigenvalue crosses $-1$ transversally.
Establishing nondegeneracy further requires showing that a cubic normal-form coefficient, involving derivatives up to third order, does not vanish; its sign determines the direction of the branch and the stability of the period-$2$ orbits~\cite{Kuznetsov2004}.

In this paper, we propose a computer-assisted approach to rigorously verify these conditions for one-parameter families of maps $x\mapsto f(x,\lambda)$ on $\R^n$.
We introduce an augmented system, the \emph{period-doubling map} $\Fpd:\R^{4n+2}\to\R^{4n+2}$ (Section~\ref{sec:Fpd}), whose unknowns are the fixed point, the critical right eigenvector, the parameter, the left eigenvector and its eigenvalue, and a center-manifold correction.
Using a contraction argument for a quasi-Newton fixed-point operator, we obtain sufficient conditions, which can be checked by computer, for the existence of a zero of $\Fpd$ in an explicit closed ball centered at a numerical approximation whose radius yields a rigorous error bound.

A benefit of this strategy is that the invertibility of $D\Fpd$ at a zero, which follows from the contraction argument, already encodes most of the required spectral information.
Indeed, an explicit determinant formula shows that $D\Fpd$ is invertible at a zero if and only if $1$ is not an eigenvalue, $-1$ is an algebraically simple eigenvalue, and it crosses the unit circle transversally along the branch of fixed points (Theorem~\ref{thm:main}).
Using the explicit error bound on the zero, we then verify via the Gershgorin Disks Theorem that no other eigenvalue lies on the unit circle.
In contrast with the augmented systems used for saddle-node and cusp bifurcations~\cite{Lessard2016,LessardPugliese2025}, the invertibility of $D\Fpd$ does not imply that the normal-form coefficient is nonzero.
We therefore compute a posteriori a rigorous enclosure of this coefficient using interval arithmetic, and verify that it excludes zero.
Together, these steps prove the existence of a nondegenerate period-doubling bifurcation and determine the direction of the branch and the stability of the period-$2$ orbits on the center manifold.

The use of computer-assisted proofs (CAPs) to study bifurcations has a long
history. In the context of period doubling, Lanford's proof of the
Feigenbaum conjectures~\cite{Lanford1982} established, using rigorous
computer estimates, the existence of a fixed point of the period-doubling
renormalization operator. More recently, Burbanks, Osbaldestin and
Thurlby~\cite{BurbanksOsbaldestinThurlby2021} obtained rigorous bounds for
this operator's fixed point and associated eigenfunctions for unimodal maps
with a critical point of degree four. Another well-known example in dynamics is
Tucker's proof that the Lorenz equations have a strange
attractor~\cite{Tucker2002}, which combines a rigorous ordinary differential equation solver with a
normal-form analysis near the origin. Such proofs often validate numerical approximations by an a~posteriori
argument~\cite{VandenBergLessard2015,DayLessardMischaikow2007}, using
interval arithmetic to control rounding errors.

For local bifurcations, augmented systems are already used in numerical
continuation to locate bifurcation points and compute normal-form coefficients~\cite{Govaerts2000,GovaertsKhoshsiarKuznetsovMeijer2007,KuznetsovMeijer2019,MooreSpence1980}.
Early interval validation methods for augmented systems were developed by Kanzawa and Oishi~\cite{KanzawaOishi1999} for several types of bifurcation
points, and by Tanaka, Murashige and Oishi~\cite{TanakaMurashigeOishi2004}
for double turning points. More recent work includes saddle-node
bifurcations~\cite{Lessard2016}, Hopf bifurcations via
desingularization~\cite{VandenBergLessardQueirolo2021}, and cusp
bifurcations~\cite{LessardPugliese2025}. Other CAPs exist for saddle-node bifurcations in regulatory network
dynamics~\cite{KepleyMischaikowQueirolo2024}, for Hopf
bifurcations~\cite{VandenBergQueirolo2022,ChurchLessard2022,ChurchQueirolo2024},
and for cocoon bifurcations~\cite{KokubuWilczakZglincinski2007}.
For discrete dynamical systems, Kamimoto, Kim, Sander and
Wanner~\cite{KamimotoKimSanderWanner2022} validated saddle-node and
Neimark--Sacker bifurcations in an age-structured red coral population model.
A different approach, based on the Conley index, was developed by
Matsue~\cite{Matsue2011} to detect bifurcations of semiflows with rigorous
numerics. In the
context of partial differential equations, rigorous numerical methods have been used to study
steady-state bifurcations~\cite{Wanner2018,Zglinczinski2015}, as well as to
continue bifurcation points and solution
manifolds~\cite{GameiroLessardPugliese2016,LessardSanderWanner2017}.
Combining the present system with such continuation methods to
follow curves of period-doubling points in two-parameter families, up to
codimension-two points such as the fold--flip
interaction~\cite{KuznetsovMeijerVanVeen2004}, is left for future work.

Of particular relevance to the present work is the CAP
of period-doubling bifurcations in the R\"ossler system
by Wilczak and Zgliczy\'nski~\cite{WilczakZglincinski2009}. There, the authors use a
Lyapunov--Schmidt reduction of the period-$2$ problem for a Poincar\'e map,
whose derivatives are enclosed by a rigorous $C^r$ integrator.
Walawska and Wilczak~\cite{WalawskaWilczak2019} developed a validation
framework for period-tupling and touch-and-go bifurcations of symmetric
periodic orbits in reversible systems, with applications to halo orbits in
the circular restricted three-body problem.
Our approach instead works directly with a square augmented system for the bifurcation point and the auxiliary vectors, requiring enclosures of $f$ and of its derivatives up to third order.
The same system therefore applies to the iterates $f^{2^k}$, and to Poincar\'e maps once these enclosures are supplied by a rigorous integrator~\cite{KapelaWilczakZgliczynski2022,Tucker2002,WilczakZglincinski2009}.

\smallskip
\paragraph{\bfseries Software resources.}
The code accompanying this article, available at~\cite{Code}, is written in Julia~\cite{Julia-2017}, using the \texttt{RadiiPolynomial}~\cite{RadiiPolynomial.jl} and \texttt{IntervalArithmetic}~\cite{IntervalArithmetic.jl} software libraries.
The figure in this article was made with \texttt{Makie}~\cite{Makie.jl}.
\smallskip

The paper is organized as follows. Section~\ref{sec:PD} recalls the bifurcation conditions and the normal-form coefficient, with proofs in Appendix~\ref{app:cm}.
Sections~\ref{sec:Fpd} and~\ref{sec:jacobian} introduce the augmented system and characterize the zeros at which its Jacobian is invertible.
Section~\ref{sec:procedure} presents the CAP procedure, which Section~\ref{sec:ks} applies to a $10$-mode Galerkin projection of the Kot--Schaffer model (Theorem~\ref{thm:ks_pd}).

\section{Period-doubling bifurcations}\label{sec:PD}

Let $f:\R^n\times\R\to\R^n$ be of class $C^3$ and consider the
parameter-dependent discrete dynamical system
\begin{equation}\label{eq:map}
x_{k+1} = f(x_k,\lambda), \qquad x_k\in\R^n,\quad \lambda\in\R.
\end{equation}
For a fixed value of $\lambda$, a fixed point satisfies $f(x,\lambda)=x$,
and its stability is governed by the eigenvalues of $\Df(x,\lambda)$.
We are interested in a period-$2$ orbit, that is, two distinct points $x_1,x_2 \in \R^n$ with $f(x_1,\lambda)=x_2$ and $f(x_2,\lambda)=x_1$, emerging from a branch of fixed points as $\lambda$ varies.
At the bifurcation, both $x_1$ and $x_2$ coincide with the fixed point, and $\Df$ has the \emph{critical} eigenvalue $-1$.
The presence of this eigenvalue alone is not sufficient for such a bifurcating family to exist.

To make this precise, letting $(x_\star,\lambda_\star)$ be a candidate bifurcation point, we use the following notation for the derivatives evaluated at this point:
\begin{equation}\label{eq:notation}
  \begin{gathered}
    A \bydef  \Df(x_\star,\lambda_\star), \qquad
    B \bydef  \Dff(x_\star,\lambda_\star), \\
    C \bydef  D_x^3f(x_\star,\lambda_\star), \qquad
    E \bydef A+I, \\
    f_\lambda \bydef  \partial_\lambda f(x_\star,\lambda_\star), \qquad
    A_\lambda \bydef  \partial_\lambda \Df(x_\star,\lambda_\star), \qquad
    B_\lambda \bydef  \partial_\lambda \Dff(x_\star,\lambda_\star).
  \end{gathered}
\end{equation}
In words, $B$ and $C$ are the symmetric bilinear and trilinear maps given by the successive derivatives of $f$ with respect to $x$, while $f_\lambda$, $A_\lambda$ and $B_\lambda$ are partial derivatives with respect to the parameter.
For $u\in\R^n$, $B(u,\cdot)$ denotes the matrix $z\mapsto B(u,z)=B(z,u)$.
We denote the spectrum of $A$ by $\sigma(A)$.

\subsection*{Bifurcation conditions}

We recall the classical conditions for a period-doubling bifurcation~\cite{Kuznetsov2004,KuznetsovMeijer2019}.
The starting point is a fixed point with a simple eigenvalue $-1$ and no other eigenvalue on the unit circle.
\begin{enumerate}[label=\textbf{(C\arabic*)}, ref=(C\arabic*), leftmargin=*]
\item\label{cond:fp} $f(x_\star,\lambda_\star) = x_\star$.
\item\label{cond:eig} $-1$ is an algebraically simple eigenvalue of $A$, with right and left eigenvectors $v_\star$ and $w_\star$, $Av_\star=-v_\star$, $A^Tw_\star=-w_\star$, normalized by $v_\star^Tv_\star=1$ and $w_\star^Tv_\star=1$.
\item\label{cond:spec} $A$ has no eigenvalue other than $-1$ on the unit circle.
\end{enumerate}
Since condition~\ref{cond:spec} excludes the eigenvalue $1$, that is, $1\notin\sigma(A)$, the Implicit Function Theorem provides a $C^3$ branch of fixed points $\lambda \mapsto x(\lambda)$ with $x(\lambda_\star)=x_\star$, and differentiating $f(x(\lambda),\lambda)=x(\lambda)$ at $\lambda_\star$ gives
\begin{equation}\label{eq:xdot}
\dot x\bydef\frac{\mathrm{d}x}{\mathrm{d}\lambda}(\lambda_\star)=-(A-I)^{-1}f_\lambda .
\end{equation}
Under~\ref{cond:eig}, the eigenvalue $-1$ then continues along this branch as a $C^2$ eigenpair $(\nuc(\lambda),v(\lambda))$ of $\Df(x(\lambda),\lambda)$ with $\nuc(\lambda_\star)=-1$ and $v(\lambda_\star)=v_\star$.
Indeed, the Implicit Function Theorem applies to $(v,\nu)\mapsto\bigl(\Df(x(\lambda),\lambda)v-\nu v,\ v^Tv-1\bigr)$ at $(v_\star,-1)$, whose Jacobian $(u,s)\mapsto(Eu-sv_\star,\ 2v_\star^Tu)$ is injective.
This is seen directly from the fact that if $Eu=sv_\star$ and $v_\star^Tu=0$, then $s=w_\star^TEu=0$ since $E^Tw_\star=0$ and $w_\star^Tv_\star=1$, so $u\in\ker E=\operatorname{span}(v_\star)$ and $v_\star^Tu=0$ gives $u=0$.
Finally, under~\ref{cond:fp}--\ref{cond:spec}, the restriction of $f(\cdot,\lambda_\star)$ to a center manifold of $x_\star$~\cite[Ch.~5]{Kuznetsov2004} is a scalar map with derivative $-1$ at the fixed point, and the following proposition, proved in Appendix~\ref{app:c}, gives the cubic coefficient of its second iterate.

\begin{proposition}\label{prop:c}
Assume~\ref{cond:fp}--\ref{cond:spec} and let
\begin{equation}\label{eq:h}
h_\star \bydef  -(A-I)^{-1}B(v_\star,v_\star).
\end{equation}
On a center manifold of $x_\star$ at $\lambda=\lambda_\star$, parametrized by a $C^3$ map
$P(\xi)=x_\star+v_\star \, \xi +O(\xi^2)$, the restricted map $g$, defined by $f(P(\xi),\lambda_\star)=P(g(\xi))$, satisfies $g(g(\xi))=\xi+c\,\xi^3+o(\xi^3)$, where
\begin{equation}\label{eq:c_formula}
c = -\tfrac{1}{3}\,w_\star^TC(v_\star,v_\star,v_\star) - w_\star^TB(v_\star,h_\star).
\end{equation}
\end{proposition}

Importantly, the sign of $c$ is intrinsic, in the sense that $c$ depends neither on the center manifold nor on the choice of $P$ in Proposition~\ref{prop:c} (Appendix~\ref{app:c}), while replacing $v_\star,w_\star$ by $\alpha v_\star,w_\star/\alpha$ only multiplies it by $\alpha^2>0$.
In the normal form $\eta\mapsto-\eta+c_{\mathrm K}\eta^3$ of the restricted map in~\cite{Kuznetsov2004}, one has $c=-2c_{\mathrm K}$, the factor $-2$ coming from the second iterate.

The two further conditions for a period-doubling bifurcation read as follows.
\begin{enumerate}[label=\textbf{(C\arabic*)}, ref=(C\arabic*), leftmargin=*, resume]
\item\label{cond:trans} The critical eigenvalue $\nuc(\lambda)$ crosses the unit circle transversally:
\[
\dnuc\bydef \frac{\mathrm{d}\nuc}{\mathrm{d}\lambda}(\lambda_\star)\neq0.
\]
\item\label{cond:nd} The normal-form coefficient $c$ in~\eqref{eq:c_formula} is nonzero.
\end{enumerate}

Under these five conditions, we say that $f$ undergoes a \emph{nondegenerate period-doubling bifurcation} at $(x_\star,\lambda_\star)$.
The following proposition, proved in Appendix~\ref{app:side}, shows that a period-$2$ orbit then bifurcates and gives its side and stability.

\begin{proposition}[The bifurcating $2$-cycle]\label{prop:side}
Assume~\ref{cond:fp}--\ref{cond:nd}, let $P(\xi,\lambda)$ be a $C^3$ parametrization of a center manifold of $(x,\lambda)\mapsto(f(x,\lambda),\lambda)$ at $(x_\star,\lambda_\star)$ with $P(0,\lambda)=x(\lambda)$ and $\partial_\xi P(0,\lambda_\star)=v_\star$, and let $g(\xi,\lambda)$ be the restricted map, so that $g(0,\lambda)=0$ and $\partial_\xi g(0,\lambda)=\nuc(\lambda)$.
For $\lambda$ near $\lambda_\star$, $g(\cdot,\lambda)$ has a period-$2$ orbit near $0$ if and only if $\dnuc(\lambda-\lambda_\star)/c>0$, that is, where $\nuc(\lambda)<-1$ if $c<0$ and where $\nuc(\lambda)>-1$ if $c>0$.
This orbit is unique, and it is attracting if and only if $c<0$.
\end{proposition}

The bifurcation is called \emph{supercritical} if $c<0$ and \emph{subcritical} if $c>0$.
In the supercritical case, the $2$-cycle is moreover attracting in $\R^n$, for $\lambda$ sufficiently close to $\lambda_\star$, if all other eigenvalues of $A$ lie inside the unit disk.
For smooth maps Proposition~\ref{prop:side} is the classical period-doubling theorem~\cite{Kuznetsov2004}; the proof in Appendix~\ref{app:side} only uses $C^3$ regularity.

\section{The period-doubling map}\label{sec:Fpd}

The conditions of Section~\ref{sec:PD} involve a fixed point, its critical
right and left eigenvectors, and the vector $h_\star$ of~\eqref{eq:h}.
We compute and validate these quantities simultaneously by writing them
as unknowns of a single system of equations. The following definition
introduces the map to which we will apply the Newton--Kantorovich Theorem.

\begin{definition}\label{def:Fpd}
The \emph{period-doubling map}
$\Fpd:\R^n\times\R^n\times\R\times \R^n\times\R \times\R^n\to\R^{4n+2}$ is
\begin{equation}\label{eq:Fpd}
\Fpd(x,v,\lambda,w,\mu,h) \bydef
\begin{pmatrix}
f(x,\lambda) - x \\
\Df(x,\lambda)\,v + v \\
v^T v - 1 \\
\Df(x,\lambda)^T w - \mu\,w \\
w^T v - 1 \\
\Df(x,\lambda)\,h - h + \Dff(x,\lambda)(v,v)
\end{pmatrix}.
\end{equation}
\end{definition}

The unknown $X=(x,v,\lambda,w,\mu,h)$ belongs to $\R^{4n+2}$, so the system is
square. The first three rows of~\eqref{eq:Fpd} form the classical defining
system of dimension $2n+1$ for a period-doubling
point~\cite{Govaerts2000,KuznetsovMeijer2019}. The remaining rows determine
$w$ and $h$, which we include to enclose the quantities needed to evaluate the
normal-form coefficient.
The unknown scalar $\mu$ is the eigenvalue paired with the left eigenvector $w$.
It is introduced to balance the number of equations and unknowns, since the left-eigenvector equation and its normalization contribute $n+1$ equations for the $n$ components of $w$.
The following lemma shows that $\mu_\star=-1$ at every zero of $\Fpd$, so that $\mu_\star$ coincides with the critical eigenvalue $\nuc(\lambda_\star)$ of~\ref{cond:trans}.

\begin{lemma}\label{lem:zero_structure}
Let $X_\star=(x_\star,v_\star,\lambda_\star,w_\star,\mu_\star,h_\star)$ be a zero of $\Fpd$. Then $\mu_\star=-1$,
$x_\star$ is a fixed point, $-1\in\sigma(A)$ with $Av_\star=-v_\star$, $A^Tw_\star=-w_\star$,
$v_\star^Tv_\star=w_\star^Tv_\star=1$, and $(A-I)h_\star=-B(v_\star,v_\star)$.
If moreover~\ref{cond:eig}--\ref{cond:spec} hold, then $h_\star$ is the vector
in~\eqref{eq:h} and~\eqref{eq:c_formula} gives the coefficient $c$ of Proposition~\ref{prop:c}. Conversely, if
$(x_\star,\lambda_\star)$ satisfies~\ref{cond:fp}--\ref{cond:spec}, then it is
the $(x,\lambda)$-component of exactly two zeros of $\Fpd$, namely
$(x_\star,\pm v_\star,\lambda_\star,\pm w_\star,-1,h_\star)$.
\end{lemma}

\begin{proof}
The first, second, third and fifth rows of~\eqref{eq:Fpd} give the fixed-point equation, $Av_\star=-v_\star$, and the two normalizations.
Multiplying the fourth row by $v_\star^T$ therefore yields
\[
0=v_\star^TA^Tw_\star-\mu_\star\,v_\star^Tw_\star=(Av_\star)^Tw_\star-\mu_\star=-1-\mu_\star,
\]
so $\mu_\star=-1$ and $A^Tw_\star=-w_\star$. The sixth row gives
$(A-I)h_\star=-B(v_\star,v_\star)$.
Under~\ref{cond:spec}, $1\notin\sigma(A)$, so $A-I$ is invertible and $h_\star$ is given by~\eqref{eq:h}.
Since~\ref{cond:fp} holds at every zero, Proposition~\ref{prop:c} applies under~\ref{cond:eig}--\ref{cond:spec} and gives~\eqref{eq:c_formula}.
Conversely, under~\ref{cond:fp}--\ref{cond:spec}, the second and third rows
give $v=\pm v_\star$ since $-1$ is simple, the fourth and fifth rows then give
$\mu=-1$ and $w=\pm w_\star$ with the same sign, and the sixth row gives
$h=h_\star$ since $A-I$ is invertible and $B(v_\star,v_\star)=B(-v_\star,-v_\star)$.
\end{proof}

\section{Invertibility of the Jacobian at a zero}\label{sec:jacobian}

A successful Newton--Kantorovich argument, discussed in Section~\ref{sec:procedure}, will provide a zero of $\Fpd$
at which the Jacobian is invertible. We now show what this implies
for the bifurcation conditions. To do so, we compute the determinant of
the Jacobian and express it in terms of the eigenvalues of $A$ and
the crossing speed of the critical eigenvalue. At a zero
$X_\star=(x_\star,v_\star,\lambda_\star,w_\star,-1,h_\star)$, the derivative
$D\Fpd(X_\star)$ maps $\hat X=(\hat x,\hat v,\hat\lambda,\hat w,\hat\mu,\hat h)$ to
\begin{equation}\label{eq:DFpd}
\begin{pmatrix}
(A-I)\hat{x} + f_\lambda\,\hat\lambda \\
B(\hat{x},v_\star) + E\hat{v} + A_\lambda v_\star\,\hat\lambda \\
2\,v_\star^T\hat{v} \\
B(\hat{x},\cdot)^Tw_\star + E^T\hat{w} - w_\star\,\hat\mu + A_\lambda^Tw_\star\,\hat\lambda \\
w_\star^T\hat{v} + \hat{w}^Tv_\star \\
B(\hat x,h_\star)+C(\hat x,v_\star,v_\star) + 2B(v_\star,\hat{v}) + (A-I)\hat{h}
+ \bigl(A_\lambda h_\star + B_\lambda(v_\star,v_\star)\bigr)\hat\lambda
\end{pmatrix},
\end{equation}
with the notation of~\eqref{eq:notation}.
The first three rows of~\eqref{eq:DFpd} involve only
$(\hat x,\hat v,\hat\lambda)$, and the next two do not involve $\hat h$.
Hence the matrix $D\Fpd(X_\star)$ is block lower triangular with diagonal
blocks $J_1$, $J_2$, $A-I$, where
\begin{equation}\label{eq:blocks}
J_1\bydef\begin{pmatrix} A-I & 0 & f_\lambda\\ B(v_\star,\cdot) & E & A_\lambda v_\star\\ 0 & 2v_\star^T & 0\end{pmatrix},\qquad
J_2\bydef\begin{pmatrix} E^T & -w_\star\\ v_\star^T & 0\end{pmatrix}.
\end{equation}
In particular,
\begin{equation}\label{eq:det_blocks}
\det D\Fpd(X_\star)=\det J_1\,\det J_2\,\det(A-I).
\end{equation}
The block $J_1$ is the Jacobian of the first three rows of~\eqref{eq:Fpd} with respect to $(x,v,\lambda)$, and $J_2$ that of the fourth and fifth rows with respect to $(w,\mu)$.

Let $-1,\nu_2,\dots,\nu_n$ be the eigenvalues of $A$, listed with algebraic
multiplicity, and set $\gamma\bydef\prod_{i=2}^n(1+\nu_i)$, with $\gamma=1$ if $n=1$ by convention. Since the nonreal $\nu_i$ come in conjugate pairs, $\gamma$ is real, and
$\gamma\neq0$ if and only if $-1$ is an algebraically simple eigenvalue of
$A$. To compute the determinants of $J_1$ and $J_2$, we use the adjugate $\operatorname{adj}(E)$ of
$E$, the transpose of its cofactor matrix, which satisfies $E\operatorname{adj}(E)=\operatorname{adj}(E)E=\det(E)I=0$,
together with the identity, obtained by expanding along the last row and column,
\begin{equation}\label{eq:bordered}
    \det\begin{pmatrix}M&p\\q^T&0\end{pmatrix}=-q^T\operatorname{adj}(M)\,p,
    \qquad M\in\Mat_{n\times n}(\R),\ p,q\in\R^n .
\end{equation}

\begin{lemma}\label{lem:adj}
Let $X_\star$ be a zero of $\Fpd$. Then $\operatorname{adj}(E)=\gamma\,v_\star w_\star^T$.
\end{lemma}

\begin{proof}
Since $0\neq v_\star\in\ker E$, either $\dim\ker E\ge2$ or $\ker E=\operatorname{span}(v_\star)$.
If $\dim\ker E\ge2$, then every minor of order $n-1$ of $E$ vanishes, so $\operatorname{adj}(E)=0$, and $\gamma=0$ since the algebraic multiplicity of $-1$ is at least two.
Otherwise $\ker E=\operatorname{span}(v_\star)$ and $\operatorname{rank}E=n-1$, so $\operatorname{adj}(E)\neq0$, and
$E\operatorname{adj}(E)=\operatorname{adj}(E)E=0$ forces its columns into
$\ker E=\operatorname{span}(v_\star)$ and its rows into
$\ker E^T=\operatorname{span}(w_\star)$; thus $\operatorname{adj}(E)=\alpha\,v_\star w_\star^T$.
Taking the trace, $\operatorname{tr}\operatorname{adj}(E)=\alpha\,w_\star^Tv_\star=\alpha$.
On the other hand, $\operatorname{tr}\operatorname{adj}(E)$ is the sum of the
principal minors of order $n-1$ of $E$, which in terms of the eigenvalues
$1+\nu_1,\dots,1+\nu_n$ of $E$, with $\nu_1=-1$, reads
\[
    \operatorname{tr}\operatorname{adj}(E)=\sum_{i=1}^n\prod_{j\neq i}(1+\nu_j)
    =\prod_{j=2}^n(1+\nu_j)=\gamma,
\]
since every term with $i\neq1$ contains the factor $1+\nu_1=0$. Hence $\alpha=\gamma$.
\end{proof}

We can now characterize the zeros at which $D\Fpd$ is invertible.

\begin{theorem}\label{thm:main}
Let $X_\star$ be a zero of $\Fpd$. Then $D\Fpd(X_\star)$ is invertible if and
only if $1\notin\sigma(A)$ and conditions~\ref{cond:eig} and~\ref{cond:trans}
hold.
Moreover, if $1\notin\sigma(A)$ and~\ref{cond:eig} holds, then the
crossing speed of~\ref{cond:trans} is
\begin{equation}\label{eq:nudot}
    \dnuc = w_\star^TB(\dot x,v_\star)+w_\star^TA_\lambda v_\star ,
\end{equation}
and
\begin{equation}\label{eq:det}
    \det D\Fpd(X_\star) = -2\,\dnuc\,\gamma^2\,\det(A-I)^2 .
\end{equation}
\end{theorem}

\begin{proof}
\smallskip\noindent\emph{The block $J_2$.}
By~\eqref{eq:blocks}, \eqref{eq:bordered} and Lemma~\ref{lem:adj},
\[
    \det J_2=-v_\star^T\operatorname{adj}(E^T)(-w_\star)=w_\star^T\operatorname{adj}(E)\,v_\star
    =\gamma\,(w_\star^Tv_\star)^2=\gamma .
\]

\smallskip\noindent\emph{Derivative of $\nuc$.}
Assume that $1\notin\sigma(A)$ and that~\ref{cond:eig} holds. As in
Section~\ref{sec:PD}, the fixed point and the eigenvalue $-1$ continue as
$x(\lambda)$ and as a $C^2$ eigenpair $(\nuc(\lambda),v(\lambda))$ of
$\Df(x(\lambda),\lambda)$ with $v(\lambda_\star)=v_\star$. Differentiating
$\Df(x(\lambda),\lambda)v(\lambda)=\nuc(\lambda)v(\lambda)$ at $\lambda_\star$ and
writing $\dot v\bydef\frac{dv}{d\lambda}(\lambda_\star)$ gives
$B(\dot x,v_\star)+A_\lambda v_\star+E\dot v=\dnuc\,v_\star$, and applying
$w_\star^T$ yields~\eqref{eq:nudot}.

\smallskip\noindent\emph{The block $J_1$.}
Under the same assumptions, adding to the last column of $J_1$
in~\eqref{eq:blocks} the first $n$ columns multiplied by the components of
$\dot x$ does not change the determinant. The entry $f_\lambda$ becomes
$f_\lambda+(A-I)\dot x=0$ by~\eqref{eq:xdot}, and the entry $A_\lambda v_\star$ becomes
$d\bydef B(\dot x,v_\star)+A_\lambda v_\star$. The resulting matrix is block
lower triangular with diagonal blocks $A-I$ and
$\bigl(\begin{smallmatrix}E&d\\2v_\star^T&0\end{smallmatrix}\bigr)$, so
by~\eqref{eq:bordered} and Lemma~\ref{lem:adj},
\[
    \det J_1=-2\det(A-I)v_\star^T\operatorname{adj}(E)d
    =-2\gamma(v_\star^Tv_\star)(w_\star^Td)\det(A-I)=-2\gamma\dnuc\det(A-I),
\]
since $v_\star^Tv_\star=1$ and $w_\star^Td=\dnuc$ by~\eqref{eq:nudot}.
Inserting both determinants in~\eqref{eq:det_blocks} proves~\eqref{eq:det}.

\smallskip\noindent\emph{Invertibility.}
By~\eqref{eq:det_blocks}, $D\Fpd(X_\star)$ is invertible only if
$\det(A-I)\neq0$ and $\det J_2=\gamma\neq0$, that is, only if
$1\notin\sigma(A)$ and~\ref{cond:eig} holds, the normalizations
of~\ref{cond:eig} being satisfied at every zero. Under these two conditions,
\eqref{eq:det} shows that $D\Fpd(X_\star)$ is invertible if and only if
$\dnuc\neq0$, which is~\ref{cond:trans}.
\end{proof}

\section{The CAP procedure}\label{sec:procedure}

We now explain how the preceding results lead to a CAP. Starting from a
numerical approximation of a zero of $\Fpd$, we first obtain an enclosure
of a zero at which the Jacobian is invertible. We then use this enclosure to verify the remaining spectral
condition~\ref{cond:spec} and enclose the normal-form coefficient $c$ of~\ref{cond:nd}.

For the first step, we use the following version of the Newton--Kantorovich Theorem;
see~\cite[Theorem~3.3]{LessardPugliese2025}. In its statement, $Q$ plays the role of an approximate inverse of $DF(\bar X)$, in practice a floating-point inverse. Throughout this
section, norms are $\infty$-norms, and $\ball{r}{\bar X}$ denotes the closed
ball of radius $r$ centered at $\bar X$.

\begin{theorem}[Newton--Kantorovich]\label{thm:NK}
Let $F:U\to\R^N$ be $C^1$ on an open set $U\subset\R^N$, let $\bar X\in U$,
let $Q\in\Mat_{N\times N}(\R)$ and $r_\star>0$ with $\ball{r_\star}{\bar X}\subset U$.
If $Y,Z\ge0$ satisfy
$\norm{Q\,F(\bar X)}\le Y$, $\sup_{X\in\ball{r_\star}{\bar X}}\norm{I-Q\,DF(X)}\le Z<1$
and $r_0\bydef  Y/(1-Z)<r_\star$, then $F$ has a unique zero $X_\star$ in $\ball{r_0}{\bar X}$, and $DF(X_\star)$ is
invertible.
\end{theorem}

\begin{remark}
The rigorous enclosures of the bounds are obtained using interval arithmetic.
In particular, the supremum over $\ball{r_\star}{\bar X}$ can be computed readily since this closed ball is the interval vector $[\bar X-r_\star,\bar X+r_\star]$.
\end{remark}

Once Theorem~\ref{thm:NK} applies, the $(x,\lambda)$-components of the ball $\ball{r_0}{\bar X}$ enclose $(x_\star,\lambda_\star)$, and we use them to verify condition~\ref{cond:spec} for $A=\Df(x_\star,\lambda_\star)$ with Gershgorin disks.
To keep the disks small, let $V$ be a matrix of approximate eigenvectors of $\Df(\bar x,\bar\lambda)$, let $\mathcal W$ be an interval matrix enclosing $V^{-1}$, and let $\mathcal M \bydef \mathcal W\,\Df([\bar x-r_0, \bar x+r_0],[\bar\lambda-r_0,\bar\lambda+r_0])\,V$ be the interval matrix obtained by interval arithmetic calculations.
Then $V^{-1}AV\in\mathcal M$, and this matrix has the same eigenvalues as $A$, with the same algebraic multiplicities.
For each $k=1, \dots, n$, let
\[
    \mathcal D_k\bydef\{z\in\C:\ |z - m_k| \le R_k\}, \qquad R_k\ge \operatorname{rad}\mathcal M_{kk}+\sum_{j\neq k}|\mathcal M_{kj}|,
\]
where $m_k$ and $\operatorname{rad}\mathcal M_{kk}$ are the midpoint and radius of $\mathcal M_{kk}$, respectively.
Each $\mathcal D_k$ contains the $k$-th Gershgorin disk of $V^{-1}AV$, and the Gershgorin Disks Theorem~\cite[Sec.~6.1]{HornJohnson2013} yields the following.

\begin{theorem}[Gershgorin]\label{thm:gershgorin}
Every eigenvalue of $A$ lies in $\bigcup_k\mathcal D_k$. If the union of
$\ell$ of the disks is disjoint from the union of the others, it contains
exactly $\ell$ eigenvalues of $A$, counted with algebraic multiplicity.
\end{theorem}

The following theorem combines these computations, whose hypotheses are checked with interval arithmetic, to verify  conditions~\ref{cond:fp}--\ref{cond:nd} of Section~\ref{sec:PD}, and thus establishes the existence of a nondegenerate period-doubling bifurcation.

\begin{theorem}[CAP procedure]\label{thm:PD}
Let $f \in C^3$ and suppose that the following hold.
\begin{enumerate}[label=\normalfont(\textbf{\arabic*}), ref=(\arabic*), leftmargin=*]
    \item\label{step:NK} Theorem~\ref{thm:NK} applies to $F=\Fpd$ at some
    $\bar X$, giving a zero $X_\star=(x_\star,v_\star,\lambda_\star,w_\star,-1,h_\star)$ in
    $\ball{r_0}{\bar X}$ with $D\Fpd(X_\star)$ invertible.
    \item\label{step:spec} Some disk $\mathcal D_{k_0}$ contains $-1$ and is
    disjoint from the union of the other disks, and each other disk lies in
    $\{|z|<1\}$ or in $\{|z|>1\}$.
    \item\label{step:NF} An interval enclosure of~\eqref{eq:c_formula} over
    $\ball{r_0}{\bar X}$ does not contain $0$.
\end{enumerate}
Then $f$ undergoes a nondegenerate period-doubling bifurcation at $(x_\star,\lambda_\star)$,
supercritical if the enclosure of $c$ is negative and subcritical if it is
positive.
\end{theorem}

\begin{proof}
By Lemma~\ref{lem:zero_structure}, \ref{cond:fp} holds and $-1$ is an
eigenvalue of $A$. Since $D\Fpd(X_\star)$ is invertible by~\ref{step:NK},
Theorem~\ref{thm:main} gives $1\notin\sigma(A)$, \ref{cond:eig}
and~\ref{cond:trans}.
By~\ref{step:spec} and Theorem~\ref{thm:gershgorin}, $\mathcal D_{k_0}$ contains
exactly one eigenvalue, namely
$-1$, and all other eigenvalues are off the unit circle; this
is~\ref{cond:spec}. By
Lemma~\ref{lem:zero_structure} and Proposition~\ref{prop:c}, \ref{step:NF}
gives~\ref{cond:nd} and the sign of $c$.
\end{proof}

\section{A Galerkin projection of the Kot--Schaffer model}\label{sec:ks}

Kot and Schaffer~\cite{KotSchaffer1986} introduced integro-difference equations
to describe populations whose growth and dispersal occur in separate phases;
a rigorous numerical study of such models is given in~\cite{DayJungeMischaikow2004}.
We consider even densities $a$ on $\mathbb{T}=\R/2\pi\mathbb{Z}$ and the map
\begin{equation}\label{eq:ks_model}
    a\mapsto\Bigl(y\mapsto\frac{1}{2\pi}\int_{-\pi}^{\pi}\beta(y-x)\,\lambda\,a(x)\bigl(1-\kappa(x)\,a(x)\bigr)\,dx\Bigr),
\end{equation}
where the growth rate $\lambda>0$ is the bifurcation parameter, $\beta$ is the
dispersal kernel, and $\kappa$ is the inverse carrying capacity.
We choose
\begin{equation*}
\kappa(x)=0.8-0.4\cos x, \qquad \beta(x) = \frac{3}{5-4 \cos x}.
\end{equation*}
Writing $a(x)=a_0+2\sum_{k\ge1}a_k\cos(kx)$, and similarly for the cosine coefficients of $\kappa$ and $\beta$, we get $\kappa_0=0.8$, $\kappa_1=-0.2$, $\kappa_k=0$ for $k\ge2$, and $\beta_k=2^{-k}$ for all $k \ge 0$.

Retaining the first cosine coefficients $k=0,\dots,n-1$, we obtain $f:\R^n\times\R\to\R^n$ given component-wise by
\begin{equation}\label{eq:ks_proj}
    f_k(a,\lambda)=\lambda\,\beta_k\Bigl(a_k-\!\!\sum_{\substack{j+p+q=k\\|j|\le1,\ |p|,|q|\le n-1}}\!\!
    \kappa_{|j|}\,a_{|p|}\,a_{|q|}\Bigr),\qquad k=0,\dots,n-1.
\end{equation}
Since $f$ is quadratic in $a$, $C\equiv0$ and thus $c=-w_\star^TB(v_\star,h_\star)$, where $B=\lambda_\star B_1$, with
\[
    [B_1(u,z)]_k\bydef-\beta_k\!\!\sum_{\substack{j+p+q=k\\|j|\le1,\ |p|,|q|\le n-1}}\!\!
    \kappa_{|j|}\bigl(u_{|p|}z_{|q|}+z_{|p|}u_{|q|}\bigr).
\]

We take $n=10$, so $\Fpd:\R^{42}\to\R^{42}$.
The proof script~\cite{Code} computes the nontrivial fixed point at $\lambda \approx 2.7$, where $\Df$ has an eigenvalue close to $-1$, takes $\bar v$ and $\bar w$ from the corresponding right and left eigenvectors and $\bar h$ from~\eqref{eq:h}, and applies Newton's method to
$\Fpd$.
Numerically, the nontrivial fixed point, which bifurcates from $a=0$ at $\lambda=1$, is stable for $1<\lambda<\lambda_\star$ and
loses stability at $\lambda_\star\approx2.7042$, where a stable $2$-cycle appears
(Figure~\ref{fig:ks_bif}).

\begin{theorem}\label{thm:ks_pd}
The map~\eqref{eq:ks_proj} with $n=10$ undergoes a supercritical
period-doubling bifurcation at a point $(a_\star,\lambda_\star)$ such that
\begin{gather*}
    \lambda_\star\in[2.70423246926594,\,2.70423246926599],\\
    c\in[-18.6289496098,\,-18.6289496097],\quad
    \dnuc\in[-0.9888090931,\,-0.9888090930].
\end{gather*}
The bifurcating $2$-cycle exists for $\lambda>\lambda_\star$ sufficiently close
to $\lambda_\star$ and is attracting. The corresponding zero of $\Fpd$ lies
within $\infty$-distance $2.1\cdot10^{-14}$ of the numerical approximation $\bar X$ computed in~\cite{Code}.
\end{theorem}

\begin{proof}
We apply the three steps of Theorem~\ref{thm:PD}. For step~\ref{step:NK},
with $r_\star=10^{-8}$, interval arithmetic gives $Y\le2.09\cdot10^{-14}$ and $Z\le1.53\cdot10^{-5}$, so Theorem~\ref{thm:NK} applies with $r_0\le2.1\cdot10^{-14}$.
For step~\ref{step:spec}, the ten Gershgorin disks over $\ball{r_0}{\bar X}$ have
radii at most $7.0\cdot10^{-12}$; the disk containing $-1$ is separated from
the others by more than $0.66$, and the other nine lie in $\{|z|\le0.337\}$.
For step~\ref{step:NF}, evaluating $-\lambda\,w^TB_1(v,h)$ over
$\ball{r_0}{\bar X}$ gives $c\in[-18.628949609792,\allowbreak\,-18.628949609788]$.
The bifurcation is therefore supercritical, and the $2$-cycle is attracting in
$\R^{10}$ since all other eigenvalues lie inside the unit disk.
Evaluating~\eqref{eq:nudot} over $\ball{r_0}{\bar X}$, with $\dot x$ enclosed
through~\eqref{eq:xdot}, gives the stated enclosure of $\dnuc$, and
Proposition~\ref{prop:side} places the $2$-cycle at $\lambda>\lambda_\star$
since $c<0$ and $\dnuc<0$. The enclosures of $\lambda_\star$ and
$a_{\star,0}$ are the corresponding components of $\ball{r_0}{\bar X}$.
\end{proof}

\begin{figure}[ht]
  \centering
  \includegraphics[width=0.86\textwidth]{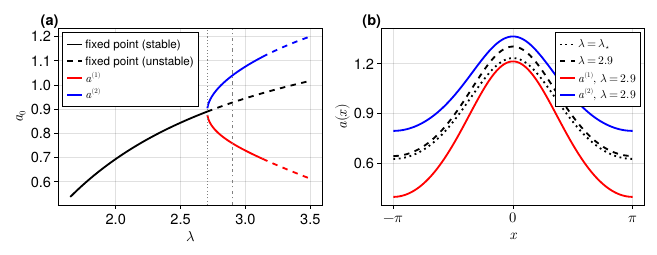}
  \caption{Numerical bifurcation diagram and profiles for the $10$-mode
Kot--Schaffer projection; only the point $(a_\star,\lambda_\star)$ is covered by the CAP. (a)~Mean density $a_0$ of the
    nontrivial fixed point  and of the two
    points $a^{(1)},a^{(2)}$ of the bifurcating $2$-cycle (solid: stable, dashed: unstable), as functions of
    $\lambda$; the vertical lines mark $\lambda_\star$ (dotted) and $\lambda=2.9$ (dash-dotted).
    (b)~Profiles $a(x)$ of the fixed point at $\lambda_\star$ (dotted) and at $\lambda=2.9$
    (dashed, unstable), and of the $2$-cycle at $\lambda=2.9$.}
  \label{fig:ks_bif}
\end{figure}

\begin{remark}\label{rem:ks_infty}
The result of Theorem~\ref{thm:ks_pd} applies to the finite-dimensional
projection~\eqref{eq:ks_proj}. Extending it to the integro-difference
operator~\eqref{eq:ks_model} requires control of the omitted Fourier modes. The Fourier coefficients of
$a_\star$, $v_\star$, $w_\star$ and $h_\star$ decay quickly ($|a_{\star,9}|\approx10^{-10}$; the slowest
decay is that of $w_\star$, with $|w_{\star,9}|\approx3\cdot10^{-8}$), which
suggests that an extension to the full model could be obtained using tail
estimates in weighted sequence spaces, as
in~\cite{VandenBergLessard2015,DayJungeMischaikow2004}.
\end{remark}


\appendix
\section{Center-manifold computations}\label{app:cm}

\subsection{Proof of Proposition~\ref{prop:c}}\label{app:c}

Throughout, \ref{cond:fp}--\ref{cond:spec} are assumed.
At $\lambda=\lambda_\star$, let $\xi\mapsto P(\xi)$ be a $C^3$ parametrization of a center manifold~\cite[Ch.~5]{Kuznetsov2004} with $P(\xi)=x_\star+v_\star \xi +O(\xi^2)$, and let $g$ be the
restricted map, defined by $f(P(\xi),\lambda_\star)=P(g(\xi))$.
Then $g$ is $C^3$ with $g(0)=0$ and we write
\[
P(\xi)=x_\star+v_\star\xi+P_2\xi^2+P_3\xi^3+o(\xi^3),\qquad
g(\xi)=a_1 \xi+a_2\xi^2+a_3\xi^3+o(\xi^3).
\]
We now expand the invariance equation $f(P(\xi),\lambda_\star)=P(g(\xi))$ to third order:
\begin{align*}
f(P(\xi),\lambda_\star) &= x_\star + A v_\star \xi + \bigl(AP_2+\tfrac12B(v_\star,v_\star)\bigr)\xi^2 +\bigl(AP_3+B(v_\star,P_2)+\tfrac16C(v_\star,v_\star,v_\star)\bigr)\xi^3+o(\xi^3), \\
P(g(\xi)) &= x_\star + a_1 v_\star \xi + \bigl(a_1^2 P_2 + a_2v_\star\bigr) \xi^2+ \bigl( a_1^3 P_3 + 2 a_1 a_2 P_2 + a_3v_\star\bigr)  \xi^3+o(\xi^3).
\end{align*}
Matching powers of $\xi$ gives
\begin{enumerate}
\item $a_1 = -1$ at order $\xi$, using $Av_\star=-v_\star$;
\item $(A-I)P_2+\tfrac12B(v_\star,v_\star)=a_2v_\star$ at order $\xi^2$, using $a_1=-1$, whence $P_2=\tfrac12 h_\star -\tfrac12a_2v_\star$ since $1\notin\sigma(A)$ and $(A-I)v_\star=-2v_\star$;
\item $(A+I)P_3+B(v_\star,P_2)+2a_2P_2+\tfrac16C(v_\star,v_\star,v_\star)=a_3v_\star$ at order $\xi^3$.
\end{enumerate}
Substituting $P_2$ gives
\[
B(v_\star,P_2)= \tfrac12 B(v_\star, h_\star) -\tfrac12 a_2 B(v_\star,v_\star).
\]
Then, using $w_\star^Tv_\star=1$ as well as $w_\star^T(A-I)=-2w_\star^T$,
\begin{align*}
w_\star^T B(v_\star,P_2) + 2a_2 w_\star^T P_2
&= \tfrac12 w_\star^T B(v_\star, h_\star) -\tfrac12 a_2 w_\star^TB(v_\star,v_\star) + a_2 w_\star^T h_\star - a_2^2 w_\star^T v_\star \\
&= \tfrac12 w_\star^T B(v_\star, h_\star) - a_2^2 .
\end{align*}
Applying $w_\star^T$ to the equation at order $\xi^3$ annihilates $(A+I)P_3$ and we obtain
\[
a_3 = \tfrac12w_\star^TB(v_\star,h_\star) - a_2^2
      +\tfrac16w_\star^TC(v_\star,v_\star,v_\star),
\]
that is,
\[
a_3+a_2^2=\tfrac12w_\star^TB(v_\star,h_\star)+\tfrac16w_\star^TC(v_\star,v_\star,v_\star).
\]
Finally, since $a_1=-1$, a direct computation gives $g(g(\xi))=\xi-2(a_3+a_2^2)\xi^3+o(\xi^3)$, so $c=-2(a_3+a_2^2)$ is given by~\eqref{eq:c_formula}.
Since it only involves $A$, $B$, $C$, $v_\star$ and $w_\star$, it depends neither on the center manifold nor on $P$.
\qed

\subsection{Proof of Proposition~\ref{prop:side}}\label{app:side}

We first show that period-$2$ orbits of $g(\cdot,\lambda)$ near $0$ correspond to nonzero fixed points of $g^2(\cdot,\lambda)\bydef g(g(\cdot,\lambda),\lambda)$.
Since $g(0,\lambda)=0$, we have
\[
g(\xi,\lambda)=\xi\int_0^1\partial_\xi g(s\xi,\lambda)\,\mathrm{d}s,
\]
and the integrand is close to $\partial_\xi g(0,\lambda_\star)=-1$ for $(\xi,\lambda)$ close to $(0,\lambda_\star)$.
Hence, for $\xi\neq0$, $g(\xi,\lambda)$ is close to $0$ and has the sign opposite to $\xi$; in particular $g(\xi,\lambda)\neq\xi$.
Therefore, if $\xi\neq0$ is a fixed point of $g^2(\cdot,\lambda)$, then $\{\xi,g(\xi,\lambda)\}$ is a period-$2$ orbit of $g(\cdot,\lambda)$, and conversely every period-$2$ orbit near $0$ consists of nonzero fixed points of $g^2(\cdot,\lambda)$.

To find these fixed points, we remove the trivial solution $\xi=0$. Since $g^2(0,\lambda)=0$, we have $g^2(\xi,\lambda)-\xi=\xi\,G(\xi,\lambda)$ with
\[
G(\xi,\lambda)\bydef\int_0^1\bigl(\partial_\xi g^2(s\xi,\lambda)-1\bigr)\,\mathrm{d}s,
\]
which is $C^2$ since $\partial_\xi g^2$ is $C^2$, and the nonzero fixed points of $g^2(\cdot,\lambda)$ are the zeros of $G(\cdot,\lambda)$ with $\xi\neq0$.
At $(0,\lambda_\star)$, we have
\begin{enumerate}
\item $G=0$ and $\partial_\lambda G=-2\dnuc$, since $G(0,\lambda)=\partial_\xi g^2(0,\lambda)-1=\nuc(\lambda)^2-1$ and $\nuc(\lambda_\star)=-1$;
\item $\partial_\xi G=0$ and $\partial_\xi^2G=2c$, since $G(\xi,\lambda_\star)=c\,\xi^2+o(\xi^2)$ by Proposition~\ref{prop:c} applied to the center manifold parametrized by $P(\cdot,\lambda_\star)$.
\end{enumerate}
Since $\dnuc\neq0$, the Implicit Function Theorem shows that, near $(0,\lambda_\star)$, the zeros of $G$ form the graph of a $C^2$ function $\xi\mapsto\lambda(\xi)$ with $\lambda(0)=\lambda_\star$.
Taking the first and second derivative of $G(\xi,\lambda(\xi))=0$ at $\xi=0$ gives $\frac{\mathrm{d}\lambda}{\mathrm{d}\xi}(0)=0$ and $\frac{\mathrm{d}^2\lambda}{\mathrm{d}\xi^2}(0)=c/\dnuc$, that is,
\[
\lambda(\xi)=\lambda_\star+\frac{c}{2\dnuc}\,\xi^2+o(\xi^2).
\]
Multiplying by $\dnuc/c$, every zero $(\xi,\lambda)$ of $G$ near $(0,\lambda_\star)$ with $\xi\neq0$ satisfies $\dnuc(\lambda-\lambda_\star)/c=\tfrac12\xi^2+o(\xi^2)>0$.
Hence $g(\cdot,\lambda)$ has no period-$2$ orbit near $0$ unless $\dnuc(\lambda-\lambda_\star)/c>0$.
Conversely, let $\lambda$ be near $\lambda_\star$ with $\dnuc(\lambda-\lambda_\star)/c>0$.
Since $\frac{\mathrm{d}\lambda}{\mathrm{d}\xi}(\xi)=\frac{c}{\dnuc}\,\xi+o(\xi)$, the function $\xi\mapsto\lambda(\xi)$ is strictly monotone on each side of $0$, so $G(\cdot,\lambda)$ has exactly one negative zero $\xi_-$ and one positive zero $\xi_+$ near $0$.
By the first paragraph, $\{\xi_+,g(\xi_+,\lambda)\}$ is a period-$2$ orbit, so $g(\xi_+,\lambda)$ is a negative zero of $G(\cdot,\lambda)$ near $0$, that is, $g(\xi_+,\lambda)=\xi_-$.
Therefore, $\{\xi_-,\xi_+\}$ is the unique period-$2$ orbit of $g(\cdot,\lambda)$ near $0$.

Finally, the stability of this orbit is determined by its eigenvalue $\partial_\xi g^2(\xi_+,\lambda)$.
Differentiating $g^2(\xi,\lambda)=\xi+\xi\,G(\xi,\lambda)$ and using $G(\xi_+,\lambda)=0$ gives
\[
\partial_\xi g^2(\xi_+,\lambda)=1+\xi_+\,\partial_\xi G(\xi_+,\lambda).
\]
Since $G$ is $C^2$ with $\partial_\xi G=0$ and $\partial_\xi^2G=2c$ at $(0,\lambda_\star)$, we have $\partial_\xi G(\xi,\lambda)=2c\,\xi+O(|\lambda-\lambda_\star|)+o(|\xi|)$, and $\lambda-\lambda_\star=O(\xi_+^2)$ on the orbit, so $\partial_\xi G(\xi_+,\lambda)=2c\,\xi_++o(\xi_+)$.
Therefore
\[
\partial_\xi g^2(\xi_+,\lambda)=1+2c\,\xi_+^2+o(\xi_+^2),
\]
and, since $c\neq0$, the orbit is attracting if and only if $c<0$.
\qed

\adresa
Olivier H\'enot, Department of Mathematics, National Taiwan University,
No.~1, Sec.~4, Roosevelt Road, Taipei 106319, Taiwan,
e-mail: \texttt{olivierhenot@ntu.edu.tw};
Jean-Philippe Lessard, Department of Mathematics and Statistics, McGill University,
805 Sherbrooke Street West, Montr\'eal, QC H3A~0B9, Canada,
e-mail: \texttt{jp.lessard@mcgill.ca}

\end{document}